\documentclass{article}
 \usepackage{inputenc}
\usepackage[francais]{babel}
\usepackage{enumerate,amssymb, amsfonts, amsmath,graphicx,amsthm}
\usepackage{lmodern}

\newtheorem{thm}{Th\'eor\`eme}
\newtheorem{prop}[thm]{Th\'eor\`eme}
\newtheorem{lem}[thm]{Lemme}
\newtheorem*{rmq}{Remarque}

\usepackage{stmaryrd}
\usepackage{dsfont}
\DeclareFontFamily{U}{txsyc}{}
\DeclareFontShape{U}{txsyc}{m}{n}{
   <-> txsyc%
}{}
\DeclareFontShape{U}{txsyc}{bx}{n}{
   <-> txbsyc%
}{}
\DeclareFontShape{U}{txsyc}{l}{n}{<->ssub * txsyc/m/n}{}
\DeclareFontShape{U}{txsyc}{b}{n}{<->ssub * txsyc/bx/n}{}
\DeclareSymbolFont{symbolsC}{U}{txsyc}{m}{n}
\SetSymbolFont{symbolsC}{bold}{U}{txsyc}{bx}{n}
\DeclareFontSubstitution{U}{txsyc}{m}{n}
\DeclareMathSymbol{\df}{\mathrel}{symbolsC}{"42}
\DeclareMathSymbol{\fd}{\mathrel}{symbolsC}{"43}
\DeclareMathSymbol{\lJoin}{\mathrel}{symbolsC}{"58}
\DeclareMathSymbol{\rJoin}{\mathrel}{symbolsC}{"59}

\newcommand{\cL}{{\cal L}}

\newcommand{\LL}{\mathbb{L}}
\newcommand{\NN}{\mathbb{N}}

\newcommand{\RR}{\mathbb{R}}

\newcommand{\TT}{\mathbb{T}}
\newcommand{\ZZ}{\mathbb{Z}}

\newcommand{\iy}{\infty}
\newcommand{\lt}{\left}

\newcommand{\na}{\nabla}

\newcommand{\ri}{\rightarrow}
\newcommand{\rt}{\right}
\newcommand{\sm}{\smallskip}

\newcommand{\fo}{\forall\ }

\newcommand{\lVe}{\lt\Vert}

\newcommand{\rVe}{\rt\Vert}

\newcommand{\un}{\mathds{1}}

\newcommand{\Vect}{\mathrm{Vect}}

\newcommand{\bq}{\begin{eqnarray*}}
\newcommand{\bqn}[1]{\begin{eqnarray}\label{#1}}
\newcommand{\eq}{\end{eqnarray*}}
\newcommand{\eqn}{\end{eqnarray}}

\newcommand{\lin}{\llbracket}
\newcommand{\rin}{\rrbracket}

\newcommand{\ttsim}{\raise.17ex\hbox{$\scriptstyle\mathtt{\sim}$}}

\title{Étude spectrale minutieuse de processus \\ moins indécis que les autres}

\author{Laurent Miclo, Pierre Monmarch\'e}
\makeindex

\begin{document}

\maketitle

\begin{abstract}
On cherche ici à quantifier la convergence à l'équilibre de processus de Markov non réversibles, en particulier en temps court. La simplicité des modèles considérés nous permet de donner une expression assez explicite de l'évolution temporelle de l'erreur $L^2$ en norme opérateur  et de la comparer avec celle des cas réversibles correspondants.
\end{abstract}

\tableofcontents

\newpage
\section{Introduction : un processus de volte-face}

Le recours à la réversibilité peut parfois limiter les performances des algorithmes stochastiques
(voir par exemple \cite{nonreversible,Neal_better,diaconis}), ce qui nous motive à mieux comprendre la convergence vers l'équilibre des processus non-réversibles.
Dans ce papier nous étudierons en détail un modèle, pour lequel on verra comment se quantifie 
le fait que les processus non-réversibles ont d'abord tendance à 
aller moins vite à l'équilibre que leur équivalent réversibles, avant 
d'atteindre des taux asymptotiques de convergence bien meilleurs.
On retrouvera notamment pour une chaîne de Markov en temps discret 
et à espace d'état fini (étudiée dans \cite{diaconis} d'un point de vue asymptotique) les phénomènes d'amorce lente de convergence  mis
en évidence dans \cite{gadat}, dans un contexte continu  d'équations d'évolutions cinétiques simples.
\par\medskip
Plus précisément, soit $(P_t)_{t\geq 0}$ un semi-groupe markovien admettant une probabilité invariante $\mu$.
Sous des conditions d'ergodicité, $P_t$  converge, en divers sens, vers $\mu$ pour de grands temps $t\geq 0$.
Considérons la convergence forte dans $L^2(\mu)$ : en interprétant $\mu$ 
 comme l'opérateur $f\mapsto (\int f d\mu) \un$, on s'intéresse à  la norme opérateur $
\lVe P_t-\mu\rVe$ dans $L^2(\mu)$.\\
Sous hypothèse de  réversibilité, le générateur $\cL$ du semi-groupe se diagonalise dans une base orthonormée (ou plus généralement, relativement à une résolution de l'identité formée d'une famille monotone de projections), ce qui permet de
 voir que 
 \bq
 \fo t\geq 0,\qquad \lVe P_t-\mu\rVe&=&\exp(-\lambda t)\eq
 où
 $-\lambda \leq 0$ est la borne supérieure du spectre de $\cL_{\vert\un^\bot}$, la restriction de $\cL$ 
 à l'espace orthogonal aux fonctions constantes dans $L^2(\mu)$
 (s'il est non nul, $\lambda$ est appelé le trou spectral de $\cL$).\\
 Dans les cas non-réversibles, il peut en être autrement, même si 
 la fonction $\RR_+\ni t\mapsto \lVe P_t-\mu\rVe$ est toujours décroissante (il s'agit d'une conséquence
 de l'inégalité de Jensen).
 Ainsi dans \cite{gadat}, pour la diffusion  constituée du couple d'un processus  d'Ornstein-Uhlenbeck linéaire et de son intégrale sur le cercle, la décroissance de $\ln( \lVe P_t-\mu\rVe)$ pour $t\geq 0$ petit
 commence par être d'ordre $t^3$.\par
 
Pour mieux appréhender ce phénomène, on va s'intéresser ici à un modèle très simple, analogue en temps continu de la marche persistante d'ordre 2 de \cite{diaconis} : une particule se déplaçant à vitesse constante sur un cercle et faisant brusquement volte-face à taux constant. Autrement dit, on considère $(Y_t)_{t\geq 0}$ un processus sur $\{-1,1\}$ qui change de signe avec un taux exponentiel $a>0$, et on pose pour tout $t\geq 0$, $X_t \df \int_0^t Y(s)ds$ sur $\mathbb{T} = \mathbb{R}/2\pi\mathbb{Z}$, de sorte que $(X_t,Y_t)$ représente le couple position-vitesse de la particule au temps $t\geq 0$. Ce modèle est cité comme exemple simple d'hypocoercivité dans \cite{jouet}. Le processus $(X_t,Y_t)_{t\geq 0}$ est caractérisé par son générateur infinitésimal, qui agit
sur des fonctions tests convenables $f$ par
\[\fo (x,y)\in\mathbb{T} \times\{-1,1\},\qquad \mathcal L_a f(x,y)\df y\partial_xf(x,y)+a\lt(f(x,-y)-f(x,y)\rt)\]
ou par le semi-groupe $(P_t^a)_{t\geq 0}$ qu'il engendre sur $L^2(\mu)$ : pour tout $f\in L^2(\mu)$,
\bq
\fo t\geq 0,\, \fo (x,y)\in\mathbb{T} \times\{-1,1\},\quad P^a_tf(x,y)\df \mathbb E \lt(f(X_t,Y_t)|X_0=x,Y_0=y\rt).\eq

La mesure invariante $\mu$ correspondante est la loi uniforme sur $\mathbb{T}\times\{-1,1\}$. 
Il est  connu que $P_t$ (pour alléger les notations,
 le paramètre  $a>0$ sera souvent sous-entendu) converge fortement dans $L^2(\mu)$
vers $\mu$ et que la vitesse finit par être exponentielle (voir la section 1.4 de \cite{jouet}, bien que le taux optimal n'y soit pas obtenu).
Comme ce serait le cas pour des opérateurs de dimension finie, on suspecte que
\bqn{lambda}\underset{t\rightarrow+\infty}{\lim}\frac1t \log\|P_t-\mu\| &=& - \lambda\eqn
avec
\bqn{lambda2}\lambda&\df& \inf\{-\mathfrak R(\theta),\ \theta\text{ valeur propre de }\mathcal L_{\vert \un^\bot}\}\eqn
On va vérifier que ceci est juste, 
mais on cherche surtout des résultats plus quantitatifs, en estimant précisément la norme $\|P_t-\mu\|$
en tout temps $t\geq 0$, car en pratique des renseignements
 asymptotiques tels que (\ref{lambda}) ne sont pas très exploitables.
Voilà l'essentiel des résultats obtenus (illustrés par les figures \ref{image1}, \ref{image2} et \ref{image3}) sur ce modèle :

\rule{\linewidth}{.5pt}

\begin{prop}\label{essentiel}
Pour $a \geq 1$, on a $\ \lambda = a - \sqrt{a^2-1}$ et pour $a\leq 1,\ \lambda = a$. Plus précisément, pour tout $t>0$,

\begin{itemize}
\item Si $a>1$ alors, en notant $\omega = \sqrt{a^2 - 1}$ et $\gamma = e^{-2\omega t}$,
\begin{eqnarray*}
 \|P_t-\mu\| & = & e^{(-a+\sqrt{a^2-1})t}\sqrt{1 + \frac{2}{\omega^2\left(\frac{1+\gamma}{1-\gamma}\right) + a\sqrt{1+\omega^2\left(\frac{1+\gamma}{1-\gamma}\right)^2} -1}}\\
\\
& = & 1 - \frac{t^3}{3} + \underset{t\rightarrow 0}{o}(t^3)\\
\\
& \underset{t\rightarrow +\infty}{\sim} & \frac{a^2}{a^2-1}e^{\lambda t}
\end{eqnarray*}

\item Si $a=1$ alors
\begin{eqnarray*}
 \|P_t-\mu\| & = & e^{-t}\sqrt{1 + \frac{2}{\sqrt{1+\frac{1}{t^2}}-1}}\\
\\
& = & 1 - \frac{t^3}{3} + \underset{t\rightarrow 0}{o}(t^3)\\
\\
& \underset{t\rightarrow +\infty}{\sim} & 2te^{-t}
\end{eqnarray*}
\item Si $a<1$ alors
\begin{eqnarray*}
 \|P_t-\mu\| & = & e^{-at}\sqrt{g(t)}\\
\\
& = & 1 - \frac{at^3}{3} + \underset{t\rightarrow 0}{o}(t^3)\\
\end{eqnarray*}
avec $g$ telle que
\begin{eqnarray*}
\underset{t\ri+\iy}{\limsup}\  g(t) & = & \frac{1+a}{1-a}\\
\\
\underset{t\ri+\iy}{\liminf}\  g(t) & = & 1
\end{eqnarray*}
et, en notant $\nu = 2\sqrt{1-a^2}$, si $t \in \left[0, \frac{\pi}{\nu}\right]$ alors

\begin{eqnarray*}
 g(t) & = & \left(1 + \frac{2}{\sqrt{\frac{\nu^2}{a^2}\frac{1}{2(1-\cos(\nu t)}+1} -1}\right)\\
\end{eqnarray*}
\end{itemize}
\end{prop}

\rule{\linewidth}{.5pt}

% \begin{figure*}[!h]
% \center
% \includegraphics[scale=0.3]{final1.png}
% \caption{Norme du semi-groupe pour différentes valeurs de $a$ au cours du temps (ici $t\in[0,1]$). Au début la décroissance est d'autant plus rapide que $a$ est grand.}\label{image1}
% \includegraphics[scale=0.3]{final2.png}
% \caption{Cependant la tendance finit par s'inverser (ici $t\in[0,3]$)}\label{image2}
% \includegraphics[scale=0.3]{final3.png}
% \caption{La meilleure vitesse asymptotique est obtenue pour $a=1$. Pour $a<1$ il arrive que la dérivée de la norme s'annule presque (ici $t\in[0,5]$)}\label{image3}
% \end{figure*}

\bigskip

Comme dans \cite{gadat}, on observe une
 décroissance initiale en $t^3$.
Dans ce contexte non-réversible, la norme opérateur $\lVe P_t-\mu\rVe$
 se comporte donc différemment du rayon spectral de $P_t-\mu$,
 qui n'est autre que $\exp(-\lambda t)$, avec $\lambda$ défini en (\ref{lambda2}).
Comme nous l'a fait remarquer le referee, 
 ceci traduit aussi 
 l'aspect anormal des opérateurs $P_t$, pour $t>0$. Par ailleurs,
  le choix optimal de $a$ (au sens du meilleur taux asymptotique de convergence exponentielle) correspond à $a=1$ et voit le facteur pré-exponentiel exploser linéairement en temps grand.

\bigskip

Le processus $(X_t,Y_t)_{t\geq 0}$ précédent est un exemple 
de processus de Markov déterministe par morceaux, famille de plus en plus étudiée dans la littérature,
notamment pour ce qui concerne les processus de type TCP (voir par exemple les articles  \cite{MR2653264,TCP}
et les références qu'ils contiennent).
Actuellement les méthodes de couplage semblent les plus efficaces pour étudier leur convergence,
au sens de la distance de Wasserstein ou de la variation totale.
Pourtant nous nous demandons si l'un au moins de ces processus, la version du TCP à taux de saut constant, 
ne pourrait pas être étudié par le biais d'une variante de l'approche spectrale que nous allons suivre dans ce papier.
En effet, il s'agit du processus sur $\RR_+$ dont le générateur $\cL$ 
agit sur des fonctions tests $f$ par
\bq\fo x\in\RR_+,\qquad\mathcal Lf(x) &\df& f'(x) + l (f(r x)-f(x)),\eq
où $l>0$ et $r\in(0,1)$ sont des constantes.
Même si la probabilité invariante associée $\mu$ est difficile à décrire explicitement,
ses moments se calculent immédiatement (en faisant agir $\cL$ sur les monômes).
La diagonalisation de $\cL$ est  facile à obtenir, car les vecteurs propres
sont des polynômes. On en déduit également une formule pour leurs produits scalaires.
On dispose donc de toute l'information spectrale
nécessaire théoriquement pour calculer les normes opérateurs.
Malheureusement nous n'avons toujours pas réussi à mener à bien les calculs.
Une autre caractéristique spectrale curieuse de $\cL$ est que bien que son spectre
soit formé de valeurs propres de multiplicité 1 et bornées par $l$, $\cL$ n'est pas borné en tant qu'opérateur dans $L^2(\mu)$, du fait de sa composante différentielle.

\bigskip

Le théorème \ref{essentiel} sera démontré au cours de la partie \ref{calcul}. La partie \ref{transition} s'attache au lien entre le modèle discret de la marche persistante et son analogue continu du volte-face. Lorsque la fréquence de changement de vitesse devient grande ce processus continu tend vers le mouvement brownien, ce qui est étudié en partie~\ref{brownien}. La partie \ref{generalisation} quant à elle discute des généralisations de ces premiers résultats à des potentiels quelconques et à la dimension supérieure. Enfin, l'appendice regroupe quelques lemmes techniques utilisés dans le reste du texte.

\section{Calcul exact de la norme}\label{calcul}

Remarquons une fois encore que si le processus était réversible, le travail serait simple puisque $\mathcal L_a$ serait diagonalisable en base orthonormée (dans $L^2(\mu)$). Ce n'est pas le cas ici mais on va tout de même pouvoir décomposer l'espace en plans stables orthogonaux ce qui nous ramènera à calculer des normes d'opérateurs en dimension 2, qu'il faudra ensuite comparer entre elles.

\begin{lem}\label{plans}
 Les plans $V_n = \{f:(x,y)\mapsto e^{inx}g(y),\ g\in\mathbb{C}^{\{-1,1\}}\}$, pour $n\in\mathbb Z$, sont invariants par $\mathcal L_a$, orthogonaux et totaux dans $L^2(\mu)$. L'action de $P_t^a$ sur $V_n$ est donnée par $e^{tK_n^{(a)}}$, où pour toute fonction test $g$, 
\bq \fo y\in\{\pm1\},\qquad K_n^{(a)} g(y)&\df& inyg(y)+a(g(-y)-g(y))\eq
(à l'instar du générateur et du semi-groupe, le paramètre $a$ sera généralement omis par la suite).
\end{lem}
\begin{proof}
 L'orthogonalité et le caractère total découlent directement de ceux de $(x \mapsto e^{inx})_{n\in\mathbb N}$ dans $L^2(\mathbb T)$. On s'assure ensuite directement que pour $f(x,y) = e^{inx}g(y)$ on a bel et bien $\mathcal L f(x,y) = e^{inx} K_ng(y)$.
\end{proof}

On est donc ramené à calculer la norme d'une matrice $2\times2$. Notons

\[R(t,a,n) \overset{def}{=} \|P_t^a - \mu\|^2_{V_n}.\]
Notons que pour tout $n\neq0$ on a $V_n \subset Ker(\mu)$. Le cas $n=0$ est un peu à part et facile à régler : $K_0$ est diagonalisable avec deux valeurs propres, 0 (associées aux constantes, que l'on retranche ici) et $-2a$. Ainsi

\[R(t,a,0) = e^{-4at}.\]
Cette restriction ne réalisera en fait jamais la norme globale (sauf $t=0$ bien sûr) : en effet on va voir que, quelque soit $a$, $\mathcal{L}$ possède des valeurs propres de parties réelles $-a$ ; ainsi sur une droite propre pour une telle valeur propre $||P_t|| = e^{-at} > e^{-2at}$. D'autre part $K_n = \bar{K}_{-n}$ et on se restreindra donc dans la suite à $n> 0$. Finalement,

\[\|P_t-\mu\| = \underset{n\geq 1}{\sup}\lt(\|P_t\|_{V_n}\rt) = \underset{n\geq 1}{\sup}\lt(\sqrt{ R(t,a,n)}\rt)\]

\subsection*{Calcul des normes des restrictions}

\begin{lem}
 Si $a>n$ alors pour tout $t>0$
\begin{eqnarray*}
R(t,a,n) & = & e^{-2(a-\sqrt{a^2 - n^2})t} \times \left( 1 + \frac{2}{\omega^2 \left(\frac{1+\gamma}{1-\gamma}\right) + \frac{a}{n}\sqrt{1+\omega^2\left(\frac{1+\gamma}{1-\gamma}\right)^2} -1}\right)
\end{eqnarray*}
avec $\omega = \sqrt{\left(\frac{a}{n}\right)^2 - 1}$ et $\gamma = e^{-2\sqrt{a^2 - n^2} t}$.
\end{lem}

\begin{proof}
 Les deux valeurs propres de $K_n$, réelles, sont $\lambda_1 = -a +n\omega > \lambda_2 = -a -n\omega$. On calcule que $(e_1,e_2)$ sont des vecteurs propres correspondants unitaires ils vérifient $|<e_1,e_2>| = \frac{n}{a}$ (les vecteurs propres sont \og d'autant plus orthogonaux\fg\ que $a$ est loin de $n$), on peut donc choisir $(e_1,e_2)$ unitaires tels que $<e_1,e_2> = \frac{n}{a}$. En posant $u=re^{i\theta}e_1+e_2$ on a ainsi

\begin{eqnarray*}
e^{tK_n}u & = & re^{i\theta}e^{\lambda_1t}e_1+e^{\lambda_2t}e_2\\
\|u\|^2 & = & r^2 + 1 + 2r\frac{n}{a}\cos(\theta)\\
\|e^{tK_n-\lambda_1 t}u\|^2 & = & r^2 + \gamma^ 2 + 2r\gamma\frac{n}{a}\cos(\theta)\\
& = & \|u\|^2 + (\gamma - 1) \times \big[\gamma + 1 + 2r\frac{n}{a}\cos(\theta)\big].
\end{eqnarray*}

En conséquence

\begin{eqnarray*}
\frac{\|e^{tK_n-\lambda_1 t}u\|^2}{\|u\|^2} & = & \frac{r^2 + \gamma^ 2 + 2r\gamma\frac{n}{a}\cos(\theta)}{r^2 + 1 + 2r\frac{n}{a}\cos(\theta)}\\
& = & \gamma + \frac{r^2 + \gamma^ 2 - \gamma r^2 - \gamma}{r^2 + 1 + 2r\frac{n}{a}\cos(\theta)}
\end{eqnarray*}
quantité qui, à $r$ fixé, est monotone en $\cos(\theta)$. Les valeurs extrémales sont donc obtenues avec $\cos(\theta) = 1$ (quitte à prendre $r<0$). On a alors

\begin{eqnarray*}
\frac{\|e^{tK_n-\lambda_1 t}u\|^2}{\|u\|^2} & = & 1 + (\gamma - 1) \times \frac{\gamma + 1 + 2r\frac{n}{a}}{r^2 + 1 + 2r\frac{n}{a}}\\
 & = & 1 - 2\frac{n}{a}(1-\gamma) \times  \frac{(r + \frac{n}{a}) - \frac{n}{a} + \frac{a}{2n}(1+\gamma)}{(r+\frac{n}{a})^2 + 1 - \left(\frac{n}{a}\right)^2}.
\end{eqnarray*}
D'après le lemme \ref{optf},
les valeurs extrêmales sont

\begin{eqnarray*}
\frac{\|e^{tK_n-\lambda_1 t}u\|^2}{\|u\|^2} & = & 1 -  \frac{\left(\frac{n}{a}\right)^2(1-\gamma) }{\left(\frac{n}{a}\right)^2 - \left(\frac{1+\gamma}{2}\right) \pm \sqrt{\left(\frac{1+\gamma}{2}\right)^2- \gamma \left(\frac{n}{a}\right)^2}}.
\end{eqnarray*}
Le maximum est obtenu pour $\pm = -$, et l'on obtient

\begin{eqnarray*}
\|e^{tK_n-\lambda_1 t}\|^2 & = & 1 +  \frac{\left(\frac{n}{a}\right)^2(1-\gamma) }{\left(\frac{1+\gamma}{2}\right)-\left(\frac{n}{a}\right)^2 + \sqrt{\left(\frac{1+\gamma}{2}\right)^2- \gamma \left(\frac{n}{a}\right)^2}}\\
\\
& = &  1 + \frac{2}{\omega^2\left(\frac{1+\gamma}{1-\gamma}\right) + \frac{a}{n}\sqrt{1+\omega^2\left(\frac{1+\gamma}{1-\gamma}\right)^2} -1}.
\end{eqnarray*}
\end{proof}

\begin{lem}
 Si $a<n$ alors pour tout $t>0$

\begin{eqnarray*}
 R(t,a,n) & = & e^{-2at} \times \left(1 + \frac{2}{\sqrt{\frac{\nu_n^2}{a^2}\frac{1}{2(1-\cos(\nu_n t))}+1} -1}\right)
\end{eqnarray*}
avec $\nu_n = 2\sqrt{n^2-a^2}$.
\end{lem}

\begin{proof}
 Dans ce cas les valeurs propres de $K_n$ sont complexes conjuguées, $\lambda_1 = \bar{\lambda}_2 = \lambda = -a + i\sqrt{n^2-a^2}$, de partie réelle $a$. On trouve des vecteurs propres normés associés $e_1$ et $e_2$ vérifiant $<e_1,e_2> = \frac{a}{n}$ (là encore le produit scalaire des vecteurs propres tend vers 0 à mesure que $a$ et $n$ s'éloignent).

 Posons $u = e_1 + r e^{i\theta}e_2$ avec $r\in\mathbb{R}$ et $\theta\in]-\pi,\pi]$. On a alors $e^{tK_n} u = e^{\lambda t}\lt(e_1 + r e^{i\theta}e^{-2i\sqrt{n^2-a^2}}e_2\rt)$, et ainsi
\begin{eqnarray*}
\|u\|^2 & = & r^2 + 1 + 2 r\frac{a}{n}\cos(\theta)\\
\|e^{tK_n - t\lambda}u\|^2 & = & r^2 + 1 + 2 r\frac{a}{n}\cos(\theta-2t\sqrt{n^2-a^2})
\end{eqnarray*}
Par le lemme \ref{optf} on obtient que le rapport entre les deux est extrémal pour $r=\pm 1$, on est donc ramené à

\[\|e^{tK_n - t\lambda}\|^2 = \underset{\theta\in\mathbb T}{\sup} \frac{\alpha_n + \cos(\theta-\nu_n t)}{\alpha_n + \cos(\theta)}\]
avec $\alpha_n = \frac{n}{a}>1$. Le lemme \ref{optg} de l'appendice conclut.
\end{proof}

\begin{lem}
 Si $a=n$ alors pour tout $t>0$
\begin{eqnarray*}
 R(t,a,n) & = & e^{-2at} \times \left(1 + \frac{2}{\sqrt{1+\frac{1}{n^2t^2}}-1}\right)
\end{eqnarray*}
\end{lem}

\begin{proof}
 Dans ce cas $-n$ est valeur propre double de $K_n$. Considérons la base $g_1(y) = 1 +iy$ et $g_2(y) = \frac{1}{n}$ de $\mathbb{C}^{\{-1,1\}}$. La matrice de $K_n$ dans cette base est alors un bloc de Jordan, d'exponentielle $e^{-nt}\begin{pmatrix}1 & t\\ 0 & 1\end{pmatrix}$. En renormalisant $g_1$ et $g_2$, on obtient des vecteurs de base unitaires $e_1$ et $e_2$ avec $<e_1,e_2> = \frac{1}{\sqrt 2}$, $e^{tK_n} e_1= e^{-nt}e_1$ et $e^{tK_n} e_2= e^{-nt}(e_2 + \sqrt 2 nte_1)$. En posant $u=(x+iy)e_1+e_2$, on a ainsi

\begin{eqnarray*}
e^{tK_n}u & = & e^{-nt}(u + \sqrt 2 nt e_1)\\
\|u\|^2 & = & x^2+y^2 + 1 + \sqrt 2 x\\
\|e^{tK_n+nt}u\|^2 & = & \|u\|^2 + 2n^2t^2 + 2\sqrt 2 nt\lt(x + \frac1{\sqrt{2}}\rt)\\
\end{eqnarray*}
Le rapport $\frac{||e^{tK_n+nt}u||^2}{||u||^2}$ est donc optimal pour $y=0$. Reste à choisir $x$.

\begin{eqnarray*}
\frac{\|e^{tK_n+nt}u\|^2}{\|u\|^2}& = & 1 + 2\sqrt 2 nt \times \frac{x + \frac1{\sqrt{2}} + \frac{nt}{\sqrt{2}}}{(x+\frac1{\sqrt{2}})^2 + \frac12}\\
\end{eqnarray*}
D'après le lemme \ref{optf}, les valeurs extrêmales sont

\begin{eqnarray*}
\frac{\|e^{tK_n+nt}u\|^2}{\|u\|^2}& = & 1 + \sqrt 2 nt \times \frac{1}{-\frac{nt}{\sqrt 2} \pm \sqrt{\frac{n^2t^2}{2}+\frac12}}\\
\end{eqnarray*}
et le maximum est obtenu pour $\pm = +$, ce qui donne le résultat escompté.
\end{proof}
Remarquons qu'on aurait pu obtenir ce résultat par continuité à partir des cas $a\lessgtr n$.

\subsection*{Comparaison des $R(t,a,n)$}

Il s'agit maintenant de comparer les normes de ces restrictions entre elles. Un développement limité en $t=0$ montre que $R(t,a,n) = 1 - \frac{n^3}{3} t^3 + o(t^3)$ pour $a\geq n$ et $R(t,a,n) = 1 - \frac{an^2}{3} t^3 + o(t^3)$ pour $a\leq n$, ce qui laisse penser qu'au moins au début $R(t,a,1)$ prévaut (autrement dit que l'erreur décroit lentement sur $V_1$ les fonctions de grande longueur d'onde en $x$). D'autre part, si $a>1$, c'est aussi sur $V_1$ que se trouve la droite propre associée à la valeur propre de $\mathcal L$ de plus grande partie réelle, c'est donc également $R(t,a,1)$ qui devrait prévaloir asymptotiquement. En fait nous allons voir que, pour l'essentiel, seule compte cette norme sur $V_1$. Notons que les expressions calculés pour $R(t,a,n)$ permettent d'étendre leur définition à $n$ non entier et qu'alors $n\in\,]0,+\infty[\ \mapsto R(t,a,n)$ est continue.

\bigskip

Dans un premier temps, on peut dériver $R(t,a,n)$ pour $n\in]0,a[$. Le lemme \ref{decroi} de l'annexe montre que cette dérivée est négative et ainsi $\underset{1\leq n<a}{\max}R(t,a,n) = R(t,a,1)$ pour tout $t>0$. Par continuité on a même $\underset{1\leq n\leq a}{\max}R(t,a,n) = R(t,a,1)$. Ainsi a-t-on réglé les cas $a\geq 1$ du théorème \ref{essentiel}, puisqu'alors $\|P_t - \mu\| = $\linebreak $\underset{n \in \mathbb Z^*}{\max}R(t,a,n) = R(t,a,1)$.

\bigskip

Le cas des $n>a$ est un peu plus délicat, pour qui 

\[R(t,a,n) = e^{-ta} \sqrt{g_n(t)}\]
avec, si $\nu_n = 2\sqrt{n^2-a^2}$,
\[g_n(t) = 1 + \frac{2}{\sqrt{\frac{\nu_n^2}{a^2}\frac{1}{2(1-\cos(\nu_n t))}+1} -1}\]

% \begin{figure}[!h]
% \center
% \includegraphics[scale=0.3]{varier.png}
% \caption{Plus $l$ est grand plus l'amplitude et la longueur d'onde de $g_l$ sont faibles.} 
% \label{varier}
% \end{figure}

\noindent qui est $2\pi/\nu_n$ périodique. Calculer le supremum des $g_n$ pour tout $t$ est à peu près impossible du fait des périodes incommensurables (cf. figure \ref{varier}). Cependant on peut penser (d'après le développement limité en 0) qu'en temps petit la norme prépondérante correspond à $n$ minimal et qu'elle le reste jusqu'à ce que $g_n$ atteigne son maximum. C'est effectivement le cas, comme on va le montrer dans un instant. Ensuite le suprémum des $g_k$ oscillera entre ce maximum et 1.

\begin{lem}
 Si $k<n$ alors pour tout $t \in \left[0,\frac{\pi}{\nu_k}\right]$ on a $g_k(t) \geq g_n(t)$.
\end{lem}
\begin{proof}
 \begin{eqnarray*}
 g_n(t) \leq g_k(t) & \Leftrightarrow & 1 + \frac{2}{\sqrt{\frac{\nu_n^2}{a^2}\frac{1}{2(1-\cos(\nu_n t))}+1} -1} \leq 1 + \frac{2}{\sqrt{\frac{\nu_k^2}{a^2}\frac{1}{2(1-\cos(\nu_k t))}+1} -1}\\
\\
& \Leftrightarrow & \frac{1-\cos(\nu_n t)}{\nu_n^2} \leq \frac{1-\cos(\nu_k t)}{\nu_k^2}
\end{eqnarray*}

Ces deux termes sont égaux et de dérivées égales en $t=0$, pour les comparer il suffit donc de comparer leurs dérivées secondes. Or, si $\nu_n \geq \nu_k$ alors $\cos(\nu_n t) \leq \cos(\nu_k t)$ pour $t\in\left[0,\frac{\pi}{\nu_n}\right]$, et donc $g_n(t) \leq g_k(t)$ pour ces $t$. Puisque $g_k$ est croissante sur $\left[0,\frac{\pi}{\nu_k}\right]$ on a pour $t \in \left[\frac{\pi}{\nu_n},\frac{\pi}{\nu_k}\right]$ 

\[g_k(t) \geq g_k(\frac{\pi}{\nu_n}) \geq g_n(\frac{\pi}{\nu_n}) \geq g_n(t).\]
On achève en constatant que $\nu_n$ est croissante en $n$.
\end{proof}

\begin{lem}
 Si $n>a$ alors pour tout $t>0$ on a $R(t,a,n) \leq R(t,a,a)$.
\end{lem}
\begin{proof}
 D'après le lemme précédent, pour tout $\varepsilon > 0$ on a $R(t,a,n) \leq R(t,a,a+\varepsilon)$ pour $t\leq \frac{\pi}{\nu_{a+\varepsilon}}$ ; or $\nu_{a+\varepsilon} \underset{\varepsilon \rightarrow 0}{\longrightarrow} 0$ et la continuité de $R$ conclut.
\end{proof}

En particulier si $a\geq 1$ pour tout $t$ on aura $||P_t-\mu|| = R(t,a,1)$, ce qui démontre les deux tiers du théorème \ref{essentiel}. Pour $a<1$ on peut comparer plus finement les $g_n$ :

\begin{lem}\label{gn}
Soit $g(t) = \underset{n\in \mathbb N}{\sup} g_n(t)$. Si $t\leq \frac{\pi}{\nu_1}$ alors $g(t) = g_1(t)$, et d'autre part
\begin{eqnarray*}
\underset{t\rightarrow+\infty}{\limsup}\  g(t) & = & \frac{1+a}{1-a}\qquad (=\sup g)\\
\\
\underset{t\rightarrow+\infty}{\liminf}\  g(t) & = & 1\qquad (=\inf g)
\end{eqnarray*}
\end{lem}
\begin{proof}
 La première assertion a déjà été démontrée, et le résultat pour la limite supérieure découle directement de la périodicité de $g_1$. Pour la limite inf, considérons $\varepsilon > 0$, et soit $N\in\mathbb N$ tel que $\frac{1+\frac{a}{N}}{1-\frac{a}{N}}\leq 1+\varepsilon$. On a ainsi, pour tout $k\geq N$ et pour tout $t>0$, $g_k(t) \leq 1+\varepsilon$. On cherche ensuite  un temps où les fonctions restantes (en nombre fini) sont simultanément proches de leur minimum. Fixons $\delta > 0$ tel que pour tout $ n < N$ et tout $ k\in \mathbb Z$, on ait
\[|t - \frac{2k\pi}{\nu_n}| \leq \delta \Rightarrow g_n(t) \leq 1 + \varepsilon.\]
Le lemme \ref{jerem} de l'appendice nous fournit
 $t\geq 1$ et des entiers $k_1,\dots,k_{N-1} \in \mathbb{N}$ tels que $|\frac{2\pi}{\nu_n} k_n -t|<\delta$ pour tout $n<N$ ; on obtient que $g_n(t) \leq 1 + \varepsilon$ pour tout $n < N$, et donc pour tout $n \in \mathbb N$. Soit $\varepsilon_0$ le minimum sur $[1/2,t+1]$ de $g-1$ (fonction continue). Si $\varepsilon_0=0$ alors $g$ est périodique et son minimum est sa limite inférieure. Sinon on peut recommencer l'argument ci-dessus pour obtenir un temps $t_2\geq 1$ tel que pour tout $n \in \mathbb N$ on ait $g_n(t_2) \leq 1 + \varepsilon_0/2$, donc nécessairement $t_2 > t +1$ ; finalement en itérant le procédé on peut trouver des temps arbitrairement grand où $g$ est arbitrairement proche de 1, ce qui conclut.
\end{proof}

Ce lemme finit de démontrer le théorème \ref{essentiel}.

\section{Du discret au continu}\label{transition}

L'étude du volte-face a initialement été motivée par celle de la marche considérée dans \cite{diaconis} : $Y_n$ est une chaîne de Markov sur $\{-1,+1\}$ qui change de signe avec probabilité $(1-\alpha)/2$, et $X^N_{n+1} = X^N_n + Y_n$ dans $\mathbb Z/N\mathbb Z \fd \mathbb Z_N$, avec $N\in\NN\setminus\{0,1\}$. Ainsi pour son $n^{\hbox{\scriptsize ième}}$ saut la particule (dont la position est $X_n^N$) persiste dans le même sens qu'au coup précédent avec une probabilité supérieure à 1/2, c'est bien l'analogue discret du processus continu des sections précédentes. Notons que la chaîne $(X_n^N)_{n\in\NN}$ est markovienne d'ordre 2.
\par
Pour peu que $N$ soit impair la chaîne est irréductible apériodique et converge donc en loi vers son unique probabilité invariante $\mu_N$, qui est la mesure uniforme sur $\mathbb Z_N\times \{\pm1\}$. L'opérateur $M_{\alpha}f(x,y) = \mathbb E\lt(f(X_1,Y_1) | X_0 = x, Y_0=y\rt)$ associé agit sur les fonctions de $L^2(\mu_N)$ et la norme d'opérateur $\|M_{\alpha}^n - \mu_N \|_{L^2(\mu_N)} \underset{n \rightarrow +\infty}{\longrightarrow} 0$ (en voyant à nouveau $\mu_N$ comme l'opérateur $f\mapsto (\int f d\mu_N)\un$). On a même
\[\underset{n\rightarrow +\infty}{\lim} \frac{1}{n}\log \left(\|M_{\alpha}^n - \mu_N \|\right) = \log(\lambda_{\alpha})\]
où, en notant $\sigma(M_{\alpha})$ le spectre de $M_{\alpha}$, $\lambda_{\alpha} = \sup (|\sigma(M_{\alpha})\smallsetminus\{1\}|)$. Ce taux exponentiel de convergence $\log(\lambda_{\alpha})$ est de valeur absolue maximale (et donc de vitesse asymptotique  la meilleure) pour $\alpha_{opt} = \frac{1-\sin(\pi/N)}{1+\sin(\pi/N)}$, pour lequel $\lambda_{opt} = \sqrt{\alpha_{opt}}$ (cf. \cite{diaconis}). En comparaison, pour la marche isotrope ($\alpha = 0$), on a $\lambda_0 = \cos(\pi/N)$. On a donc amélioré la convergence en temps long car
\[\cos(\pi/N) = \sqrt{(1-\sin(\pi/N))(1+\sin(\pi/N))}\geq \sqrt{\frac{1-\sin(\pi/N)}{1+\sin(\pi/N)}}.\]
L'étude du volte-face a permis de mieux comprendre l'amorce de convergence en temps petit, et nous pouvons maintenant faire le lien avec la marche discrète. D'abord constatons que des calculs identiques aux précédents nous permettent de calculer la norme de $M$. Pour $k\in\lin1,N\rin$ on notera $e^{2ik\pi/N} = C_k + i S_k$, $\alpha_l  = \frac{1-|S_l|}{1+|S_l|}$, $C_0^2 = \frac{4\alpha}{(1+\alpha)^2}$ et $S_0^2 =  \lt(\frac{1-\alpha}{1+\alpha}\rt)^2$.

\begin{lem}\label{caldis}
Les plans $W_k = \{(x,y)\mapsto e^{2ik\pi x/N} g(y), g\in\mathbb C^{\pm1}\}$ sont stables par $M$. Notons $R_N(n,\alpha,k) \overset{\mathrm{def}}{=} \|M_\alpha^n - \mu\|_{W_k}^2$.
\begin{itemize}
\item si $\alpha < \alpha_k$ alors
\begin{eqnarray*}
R_N(n,\alpha,k) & = & \lambda_+^{2n} \times \left( 1 + \frac{2}{\omega^2\left(\frac{1+\gamma}{1-\gamma}\right) + \frac{S_0}{S_k}\sqrt{1+\omega^2\left(\frac{1+\gamma}{1-\gamma}\right)^2} -1}\right)
\end{eqnarray*}
avec $\lambda_\pm = \sqrt{\alpha}\left(\frac{|C_k|}{C_0} \pm \sqrt{\left(\frac{C_k}{C_0}\right)^2-1}\right)$, $\gamma = \left(\frac{\lambda_-}{\lambda_+}\right)^n$ et $\omega^2 = \left(\frac{S_0}{S_k}\right)^2-1$.

\item si $\alpha > \alpha_k$ alors 
\begin{eqnarray*}
 R_N(n,\alpha,k) & = & \alpha^n \times \left(1 + \frac{2}{\sqrt{2\frac{\left(\frac{S_k}{S_0}\right)^2-1}{1-\cos(2n\psi)}+1} -1}\right)
\end{eqnarray*}
où $\tan\psi = \sqrt{\left(\frac{C_0}{C_k}\right)^2-1}$.%\sqrt{\frac{S_k^2 - S_0^2}{1- S_k^2}}$.

\item si enfin $\alpha = \alpha_k$ alors
\begin{eqnarray*}
 R_N(n,\alpha,k) & = & \alpha^n \times \left(1 + \frac{2}{\sqrt{1+\frac{C_0^2}{S_0^2n^2}}-1}\right)
\end{eqnarray*}
\end{itemize}
\end{lem}

\begin{proof}
La démarche et les calculs sont quasiment les mêmes que dans le cas continu et n'amènent aucune difficulté nouvelle.
\end{proof}

Lorsqu'on veut passer du modèle discret au continu, plutôt que $X_n^N \in \mathbb Z_N$ il vaut mieux regarder $U^N_t = \frac{2\pi}{N}X^N_n \in \mathbb T$ si $t = n \frac{2\pi}{N}$ que l'on prolonge de façon affine à $t\geq 0$ et $V^N_t = Y_n$ si $t\in\frac{2\pi}{N}[n,n+1[$. Si la probabilité de changer de sens $\frac{1-\alpha_N}{2}$ est de l'ordre de $\frac{1}{N}$, la convergence des temps entre deux changements vers une loi exponentielle donne la convergence en loi de $(U^N,V^N)$ vers le processus continu. Remarquons que pour $u = \frac{2\pi}{N} x$ on peut réécrire $e^{i\frac{2k\pi}{N}x} = e^{iku}$, l'espace $V_k$ correspond donc à $W_k$ :

\begin{lem}
 Pour tout $t>0$ et $k\in\mathbb Z$, si $\alpha^{(N)}\in[0,1]$ est tel que $\frac{N}{2\pi} \times \frac{1-\alpha^{(N)}}{2} \underset{N\rightarrow +\infty}{\longrightarrow} a$ alors
\[R_N\lt(\lt\lfloor \frac{Nt}{2\pi}\rt\rfloor,\alpha^{(N)},k\rt) \underset{N\rightarrow +\infty}{\longrightarrow} R(t,a,k).\]
\end{lem}

\begin{proof}
 On le vérifie sans difficulté particulière sur les expressions analytique données dans le lemme \ref{caldis} et la partie \ref{calcul}.
\end{proof}

Cependant, contrairement au cas continu, dans la marche discrète la plus grande valeur propre (associée au $|\cos(\frac{2k\pi}{N})|$ maximal) ne correspond pas à $k=1$ mais à $k = \pm\lt\lfloor\frac{N}{2}\rt\rfloor$. Pour avoir la convergence des normes globales d'opérateurs il faut ignorer les deux plans $W_{\pm\lt\lfloor\frac{N}{2}\rt\rfloor}$. En un sens le caractère fini des positions prises par la particule entraîne l'existence d'observables qui convergent mal, ce qui disparaît à la limite des processus, mais pas dans le passage à la limite des normes.

\begin{lem}
  Pour tout $t>0$ et $k\in\mathbb Z$, si $\alpha^{(N)}\in[0,1]$ est tel que $\frac{N}{2\pi} \times \frac{1-\alpha^{(N)}}{2} \underset{N\rightarrow +\infty}{\longrightarrow} a$ alors
\[R_N\lt(\lt\lfloor \frac{Nt}{2\pi}\rt\rfloor,\alpha^{(N)},\lt\lfloor\frac{N}{2}\rt\rfloor - k\rt) \underset{N\rightarrow +\infty}{\longrightarrow} R\lt(t,a,k+\frac{1}{2}\rt).\]
\end{lem}

\begin{proof}
Les calculs sont les mêmes que précédemment ; le $1/2$ apparaît avec
\[\sin\lt(\frac{2\pi}{N}\lt(\lt\lfloor\frac{N}{2}\rt\rfloor-k\rt)\rt) = \sin\lt(\pi - \frac{2\pi}{N}\lt(\lt\lfloor\frac{N}{2}\rt\rfloor-k\rt)\rt) = \sin\lt(\frac{2\pi}{N}\lt(k+\frac12\rt)\rt)\]
\end{proof}

Le travail de comparaison des $R(t,a,n)$ englobait déjà les $n$ non-entiers, et en notant pour tout $t\geq 0$,
\bq
\fo u\in \ZZ_N,\,\fo v\in\{\pm1\},\qquad
P^N_t f(u,v) &\df& \mathbb E(f(X^N_n,Y^N_n)|X^N_0 = u, Y^N_0 = v)\eq
avec $n=\lfloor Nt/(2\pi)\rfloor$,
on obtient \emph{in fine}

\rule{\linewidth}{.5pt}

\begin{prop}
Si $\frac{N}{2\pi} \times \frac{1-\alpha^{(N)}}{2} \underset{N\rightarrow +\infty}{\longrightarrow} a\geq\frac12$ alors
\[||P^N_t - \mu_N|| \underset{N\rightarrow +\infty}{\longrightarrow} R\lt(t,a,\frac12\rt).\]
D'autre part si l'on note $\mathcal V_N \df \Vect( W_{\lfloor{N}/{2}\rfloor},W_{-\lfloor{N}/{2}\rfloor})^\perp$ et si $a\geq1$ alors
\[||P^N_t - \mu_N||_{\mathcal V_N} \underset{N\rightarrow +\infty}{\longrightarrow} ||P^a_t - \mu||. \]
Les convergences sont uniformes en $t$.
\end{prop}

\rule{\linewidth}{.5pt}

\begin{proof}
 Tout est déjà démontré sauf le caractère uniforme en $t$ ; les fonctions en présence étant toutes décroissantes et les limites continues, il découle du théorème de Dini.
\end{proof}

\begin{rmq}
 Notons que, grosso modo, les choses se passent bien également pour $a<\frac12$ dans le  premier cas et pour $a<1$ dans le second mais avec de très légères subtilités : par exemple, dans le deuxième cas et pour reprendre les notations de la partie \ref{calcul}, la fonction $g(t)$ limite n'est pas le supremum des $g_n(t)$ pour $n$ entier mais pour $n$ entier ou demi-entier, ce qui peut éventuellement légèrement changer la valeur exacte de la norme lors d'un \og creux \fg\ de $R(t,a,1)$.
\end{rmq}

Un constat particulier sur ce défaut de convergence du discret vers le continu est que si l'on prend pour tout $N$ la probabilité optimale (au sens du trou spectral maximal) de changer de sens dans la marche persistante, alors on converge vers un taux $1/2$ de saut pour $Y_t$, qui n'est pas optimal pour le processus continu, et qui donne le même taux exponentiel $1/2$ de convergence que le mouvement brownien sur le tore.

Cependant le phénomène de décroissance initiale en $t^3$, lui, n'est pas affecté par cette subtilité ; c'est normal car son origine n'est pas dans la prise du supremum des normes des restrictions mais, déjà localement, sur chacun des plans $W_k$. Une interprétation possible est que prendre, au lieu d'un processus réversible, l'intégrale d'un processus réversible retarde initialement l'effet de mélange du hasard ; ou bien que la particule commence par se déplacer de façon déterministe et brouille donc moins bien les pistes qu'une diffusion au moins initialement.

Si pour $N$ grand, on compare (en oubliant le défaut de convergence et les fonctions de $\Vect( W_{\lfloor{N}/{2}\rfloor},W_{-\lfloor{N}/{2}\rfloor})$) la marche simple et la marche persistante pour $a=1$ à la limite, pour un nombre $n$ d'itérations fixé, l'écart $L^2$ à l'équilibre de la marche réversible est environ $1 - \frac{t}{2}$ avec $t = n \left(\frac{2\pi}{N}\right)^2$ (si cette quantité est petite) et celle de la marche persistante est $1 - \frac{t^3}{3}$ avec $t = n\frac{2\pi}{N}$ (si $n\ll N$), qui devient meilleure que la précédente pour $n \approx \sqrt{\frac{3}{4\pi}N}$ (qui assure aussi la validité des asymptotiques précédentes): c'est le nombre d'itérations à partir duquel la marche d'ordre 2 est plus proche de la mesure uniforme que la réversible.

\section{Du continu au mouvement brownien}\label{brownien}

Lorsque $a \rightarrow +\infty$, la vitesse du processus continu saute de plus en plus vite de $-1$ en $1$ ; à la limite, les vitesses en deux temps distincts devraient donc être décorrelées. Le processus devrait en conséquence être l'intégrale d'un bruit blanc, autrement dit un mouvement brownien. Avec la bonne renormalisation, c'est effectivement le cas :

\rule{\linewidth}{.5pt}

\begin{prop}
$X^a = (X_{ta})_{t>0}$ converge en loi vers un mouvement brownien standard sur $\mathbb T$ quand $a\rightarrow +\infty$.
\end{prop} 

\rule{\linewidth}{.5pt}

\begin{proof}
Notons $\tilde{Y}_t = (-1)^{N_t}$ où $N_t$ est un processus de de Poisson de paramètre 1. Ainsi $X$ suit la même loi que $\int_0^. \tilde{Y}_{as}ds$
\begin{eqnarray*}
 X^a(t) & \overset{\mathcal L}{=} & \int_0^{ta} \tilde{Y}_{as}ds\\
& = & \frac1a\int_0^{ta^2} \tilde{Y}_{u}du
\end{eqnarray*}
ce qui nous ramène à l'exemple 3 p. 360 de \cite{convergence} où l'on nous indique la marche à suivre.

\bigskip

Détaillons : on montre d'abord que $M(t) = \tilde{Y}_t + 2\int_0^t \tilde{Y}_u du$ est une martingale. En effet le nombre de changement de signes de $\tilde{Y}_t$ dans une période $t-s$ suit une loi de Poisson de paramètre $t-s$, et ainsi

\begin{eqnarray*}
 \mathbb P(\tilde{Y}_t  =  \tilde{Y}_s) & = & \sum_{k \text{ pair}} \frac{(t-s)^k}{k!}e^{-(t-s)}  =  \cosh(t-s)e^{-(t-s)}\\
\mathbb P(\tilde{Y}_t  =  -\tilde{Y}_s) & = & \sum_{k \text{ impair}} \frac{(t-s)^k}{k!}e^{-(t-s)}  =  \sinh(t-s)e^{-(t-s)}
\end{eqnarray*}

Ainsi $\mathbb E(\tilde{Y}_t|\mathcal F_s) = \tilde{Y}_s e^{-2(t-s)}$ et
\begin{eqnarray*}
 \mathbb E(M(t)|\mathcal F_s) & = &  \tilde{Y}_s e^{-2(t-s)} + 2 \int_0^s \tilde{Y}_udu + 2 \int_s^t \tilde{Y}_s e^{-2(u-s)}du\\
& = & \tilde{Y}_s + 2 \int_0^s \tilde{Y}_udu\\
& = & M(s)
\end{eqnarray*}

Si l'on montre la convergence de la martingale $\frac1nM(n^2t) = 2 X^n_t +  \frac1n\tilde{Y}_t$ vers le brownien, on aura celle de $X^n$ ; or la première s'obtient de la convergence des crochets. La variation quadratique de $\int_0^s \tilde{Y}_udu$, processus 1-lipschitzien, est nulle, donc
\[<M>_t = \underset{\delta\rightarrow0}{\overset{\mathbb P}{\lim}} \sum_{t_i\in\pi}(\tilde{Y}_{t_{i+1}} - \tilde{Y}_{t_i})^2\]
où la limite en proba a lieu lorsque le pas $\delta$ de la partition $\pi$ de $[0,t]$ tend vers 0. Notons $Z_t$ le nombre de saut de $\tilde{Y}$ sur cet intervalle.

\begin{eqnarray*}
 \mathbb P(\sum_{t_i\in\pi}(\tilde{Y}_{t_{i+1}} - \tilde{Y}_{t_i})^2 \neq 4Z_t) & \leq & \mathbb P(\text{deux sauts sont distants de moins de }\delta)\\
& \underset{\delta\rightarrow0} \rightarrow & 0
\end{eqnarray*}

Ainsi $<M>_t = 4 Z_t$, et $< \frac1n M(n^2.) >_t = \frac{4}{n^2} Z_{n^2 t} \underset{n\rightarrow +\infty}{\longrightarrow} 4t$ (par la loi des grands nombres), ce qui donne la convergence de $\frac12 \frac1n M(n^2t)$ (et donc de $X^n_t$) vers le mouvement brownien standard (cf \cite{convergence}).
\end{proof}
\bigskip

Qu'en est-il de la norme ? Celle du modèle irréversible converge-t-elle vers celle du brownien ? C'est effectivement le cas. Le générateur du mouvement brownien est $\frac12\partial_x^2$, diagonalisable dans la base orthonormée des $x\mapsto e^{inx}$ pour les valeurs propres $-\frac{n^2}{2}$. Rappelons la norme du semi-groupe associé à $(X_t,Y_t)$ sur le plan $V_n$, quand $a>n$ :

\begin{eqnarray*}
\|P_t\|_{V_n}^2 & = &e^{2\lambda_1 t}\left( 1 + \frac{2}{\frac{\omega^2}{n^2}\left(\frac{1+\gamma}{1-\gamma}\right) + \frac{a}{n}\sqrt{1+\frac{\omega^2}{n^2}\left(\frac{1+\gamma}{1-\gamma}\right)^2} -1}\right)
\end{eqnarray*}

\noindent avec $\lambda_1 = -a+\sqrt{a^2-n^2}$, $\omega = \sqrt{a^2 - n^2}$ et $\gamma = e^{-2t\sqrt{a^2-n^2}}$. On observe que $a\lambda_1 \rightarrow -\frac12n^2$, $\omega \rightarrow +\infty$ et que $\gamma^a\rightarrow 0$ quand $a\rightarrow +\infty$ ; Au final, en notant $P^a_t$ le semi-groupe associé à $(X_{at},Y_{at})$, on récupère

\[\|P^a_t\|_{V_n} \underset{a\rightarrow +\infty}{\longrightarrow} e^{-\frac12n^2t},\]

\noindent ce qui est la norme du semi-groupe $Q_t$ associé au mouvement brownien sur la droite $\Vect\{x\mapsto e^{inx}\}$. En particulier la convergence pour $n=1$ donne la convergence de la norme globale $\|P^a_t -\mu\| \longrightarrow \|Q_t - \lambda\|$.

\section{Généralisations}\label{generalisation}

\subsection{Avec un potentiel général}

En fait le cas précédent, où la mesure invariante pour $X_t$ est la loi uniforme sur le cercle, est immédiatement généralisable à des processus admettant pour loi limite n'importe quelle mesure de la forme $\nu=e^{-V(x)}dx/(2\pi)$, où le potentiel $V$ est supposé normalisé de sorte que $\nu(\TT)=1$. En effet, considérons comme précédemment $Y_t\in\{-1,1\}$ qui, avec taux $a$, change de signe. Soit $X_t\in\mathbb{T}$ la solution de
\bqn{XY}dX_t = Y_te^{V(X_t)}dt.\eqn
Autrement dit $X_t$ représente la position d'une particule se déplaçant à vitesse (déterministe) inversement proportionnelle à la densité $e^{-V(x)}$ (les zones \og peu intéressantes\fg\ sont parcourues plus vite) et changeant de sens de parcours selon des temps exponentiels . Montrons qu'alors la mesure invariante pour $(X_t,Y_t)$ est $\mu=\nu\otimes\mathcal U_{\{-1,1\}}$, et que la norme 2 du semi-groupe associé se calcule exactement comme précédemment. Le générateur markovien associé au processus est

\[\mathcal{L}f(x,y) = e^{V(x)}y\partial_xf(x,y) + a\lt(f(x,-y)-f(x,y)\rt)\]
Et l'on vérifie

\begin{eqnarray*}
\lefteqn{\nu\otimes\mathcal U_{\{-1,1\}}\big[\mathcal{L}f(x,y)\big]} & &\\
\\
& = &\int_{x\in\mathbb{T}}\int_{y=\pm}\big[e^{V(x)}y\partial_xf(x,y) + a\lt(f(x,-y)-f(x,y)\rt)\big]e^{-V(x)}dxdy\\
\\
& = & \int_{y=\pm} y \left(\int_{x\in\mathbb{T}}\partial_xf(x,y)dx\right)dy\\
\\
& = & 0.
\end{eqnarray*}
Considérons pour $n\in\mathbb N$, $g_n(x) = \exp\left( in\int_0^xe^{-V(u)}du\right)$ (on a bien $g_n(0)=g_n(2\pi)$ de par la normalisation de $V$) et des fonctions de la forme $f(x,y) = g_n(x)h(y)$. On a alors

\bqn{gKa}
\nonumber \mathcal{L}f(x,y) & = & g_n(x)\big[ inyh(y)+a(h(-y)-h(y)\big]\\
& = & g_n(x) K_n^{(a)}h(y)
\eqn
où $K_n^{(a)}$  a été  défini dans le lemme \ref{plans} pour le cas uniforme. On parvient donc là encore à décomposer l'espace en plans stables $V_n$, et ces plans sont à nouveau orthogonaux entre eux dans $L^2(\mu)$ :

\begin{eqnarray*}
<g_n,g_k>_{L^2(\mu)} & = & \int_0^{2\pi} \exp\left( i(n-k)\int_0^xe^{-V(u)}du\right)e^{-V(x)}dx\\
& = & \int_0^{2\pi} e^{ i(n-k)u}du\\
& = & 2\pi \delta_{nk}
\end{eqnarray*}

Finalement, si $P^V_t$ est le semi-groupe associé au processus (et $P_t$ est toujours celui associé au potentiel nul), on a exactement
\[||P^V_t - \mu||_{L^2(\mu)} = ||P_t - \lambda\otimes\mathcal U_{\{-1,1\}}||_{L^2(\lambda\otimes\mathcal U_{\{-1,1\}})}\]

D'après la section 1, le meilleur taux de convergence asymptotique est donc obtenu en choisissant $a=1$.
Remarquons que lorsque $V$ n'est connu qu'à une constante additive près
et que l'on veut garder le bénéfice de l'écriture (\ref{XY}), il faut modifier en conséquence
la définition de $\nu$ et des $g_n$, pour $n\in\NN$,
et on doit remplacer $K_n^{(a)}$ par $Z^{-1}K_n^{(aZ)}$, avec
$Z\df \int_0^{2\pi}e^{-V(x)}\,dx/(2\pi)$ dans (\ref{gKa}).
Le choix optimal de $a$ est alors $Z^{-1}$, qui malheureusement n'est pas connu en pratique.

\subsection{Remarque sur les dimensions supérieures}

Remarquons que, dans l'optique d'un algorithme de Monte-Carlo non réversible, les résultats s'adaptent à la dimension supérieure. Ainsi en définissant $Y_t^1,\dots,Y_t^d$ et $X_t^1,\dots,X_t^d$ comme précédemment, dans le cas où $V(x) = \sum V_i(x_i)$, on construit un semi-groupe $P^V_t$ sur $\mathbb T^d$ de mesure invariante $\mu$ proportionnelle à $e^{-V(x)}dx\otimes \mathcal U_{\{-1,1\}}^{\otimes d}$ et de norme

\[||P^V_t||_{L^2(\nu)} = \prod_{i=1}^d||P^{V_i}_t||_{L^2(Z_ie^{-V_i(x_i)}dx_i\otimes U_{\{-1,1\}})}\] 
où les $Z_i$, $i\in\lin 1,d\rin$, sont les constantes de normalisation.
On aurait pu imaginer un autre processus, construit en gardant l'idée d'une particule dont la vitesse scalaire dépendrait de façon déterministe de la position mais dont la direction changerait aléatoirement à taux constant. Cela donnerait un générateur du type :
\begin{eqnarray*}
\mathcal{L}f(x,y) & = & e^{V(x)}\na_xf(x,y).y + a \int_{\mathbb{S}^d}\lt(f(x,z)-f(x,y)\rt)dz\\
\end{eqnarray*}
pour des fonctions tests $f$ régulières. Ci-dessus les vitesses sont prises uniformément sur la sphère mais on aurait pu les choisir différemment sans que les remarques à suivre ne s'en trouvent modifiées. La mesure invariante est alors $Ze^{-V(x)}dx\otimes\mathcal U_{\mathbb S^d}$, avec $Z=Z_1\cdots Z_d$, ce qui semble bien parti. Néanmoins, à part pour un potentiel nul, on ne va pas pouvoir se ramener à l'étude d'un opérateur sur les vitesses par la même méthode qu'avant, c'est-à-dire en trouvant des fonctions propres de la famille d'opérateurs $K_y : f(x)\mapsto e^{V(x)}\na f(x).y$ sous la forme $f(x)=e^{u(x)}$, qui permettaient jusqu'ici de se ramener à des opérateurs n'agissant que sur les vitesses. En effet on a alors
\[K_yf(x) = e^{V(x)}f(x)\na u(x).y\]
Il s'agirait donc de trouver une fonction $u : \mathbb R^d \rightarrow \mathbb R$ de différentielle $x\mapsto e^{-V(x)}(c_1 dx_1+c_2dx_2+\dots+c_ddx_d)$, avec $c_1, ..., c_d$ des constantes. Or si $d>1$, cette 1-forme linéaire n'est pas exacte (n'étant pas fermée), et un tel $u$ ne saurait exister. En comparaison, pour le processus avec des coordonnées indépendantes du début de ce paragraphe, la 1-forme linéaire qui apparaît est $x\mapsto \sum e^{-V_i(x_i)}dx_i$, qui est bel et bien exacte.

\section{Appendice}

\begin{lem}\label{optf}
Si $f(R) = \frac{R-a}{R^2 + b}$ avec $b>0$, alors $f$ admet ses valeurs extrémales en $R_\pm = a \pm \sqrt{a^2 + b}$, et ces valeurs sont $f(R_\pm) = \frac{1}{2R_\pm}$.
\end{lem}

\begin{lem}\label{optg}
Si $g(\theta) = \frac{\alpha + \cos(\theta-s)}{\alpha + \cos(\theta)}$ avec $\alpha > 1$, alors

\[\underset{\theta\in\mathbb T}{\max}\ g(\theta) = 1 + \frac{2}{\sqrt{\frac{2(\alpha^2 - 1)}{1 - \cos(s)}+1} - 1}\]
De plus ce maximum est majoré par $\frac{\alpha+1}{\alpha-1}$, borne atteinte uniquement pour $s = \pi\ [2\pi]$.
\end{lem}

\begin{proof}
Le premier lemme ne présente aucune difficulté. Pour le second, remarquons tout d'abord pour $s=0\ [2\pi]$ que $g$ est alors constante égale à 1 et son max l'est également, le lemme est donc vrai dans ce cas. Supposons dans la suite que $1-\cos(s) \neq 0$. Réécrivons maintenant
\[g(\theta) = \frac{\alpha + \cos(\theta +s)}{\alpha + \cos(\theta)} = \cos(s) + \frac{\alpha (1 -  \cos (s)) - \sin(\theta)\sin(s)}{\alpha + \cos(\theta)}\]
$g(\theta)$ étant continue périodique il suffit de déterminer ses points critiques. Or $g'(\theta)=0$ équivaut à
\begin{eqnarray*}
0 & = & - \cos(\theta)\sin(s)\lt(\alpha + \cos(\theta)\rt) + \sin(\theta)\lt(\alpha(1-\cos(s))-\sin(\theta)\sin(s)\rt)\\
& = & -\sin(s) + \sin(\theta)\lt(\alpha(1-\cos(s)\rt) - \alpha\cos(\theta)\sin(s),
\end{eqnarray*}
équation affine dont les solutions sont
\[\begin{pmatrix}\cos(\theta)\\ \sin(\theta)\end{pmatrix} = \begin{pmatrix} - \beta\\ 0\end{pmatrix} + t \begin{pmatrix}1 - \cos(s)\\ \sin(s)\end{pmatrix}\]
pour  $t\in\mathbb{R}$ et où l'on note $\beta = \frac{1}{\alpha}$. La condition $\cos^2 + \sin^2 = 1$ équivaut à

\[t^2 - \beta t + \frac{\beta^2 - 1}{1 - \cos(s)}\]
qui admet nécessairement deux solutions réelles puisque $g$ est périodique non constante donc possède au moins deux points critiques. Ces solutions sont données par

\[t (1 - \cos(s)) = \frac12\beta(1 - \cos(s)) + \frac12\varepsilon\sqrt{2(1 - \cos(s)) - \beta^2\sin^2(s)}\]
où $\varepsilon = \pm 1$. On obtient ainsi les valeurs extrêmales de $g$ :
\begin{eqnarray*}
g(\theta_\varepsilon) & = & \cos(s) + \frac{\alpha (1 -  \cos (s)) - t\sin^2(s)}{\alpha - \beta + t(1 - \cos(s))}\\
\\
& = & \frac{\alpha - t\sin^2(s) - \beta \cos(s) + t\cos(s) - t\cos^(s)}{\alpha - \beta + t(1 - \cos(s))}\\
\\
& = &  \frac{\alpha - \beta \cos(s) - t (1 - \cos(s))}{\alpha - \beta + t(1 - \cos(s))}\\
\\
& = & \frac{\alpha - \frac12\beta(1+ \cos(s)) - \frac12\varepsilon\sqrt{2(1 - \cos(s)) - \beta^2\sin^2(s)}}{\alpha - \frac12\beta(1+ \cos(s))  + \frac12\varepsilon\sqrt{2(1 - \cos(s)) - \beta^2\sin^2(s)}}\\
\end{eqnarray*}
Puisque $\alpha > 1 > \beta$, on a $\alpha - \frac12\beta(1+ \cos(s)) > 0$ et la valeur ci-dessus est maximale pour $\varepsilon = -1$, et ainsi
\begin{align*}
\underset{\theta\in\mathbb T}{\max} g(\theta) & =  \frac{\alpha - \beta\left(\frac{1+ \cos(s)}{2}\right) + \sqrt{\left(\frac{1 - \cos(s)}{2}\right)\lt(1 - \beta^2\frac{1+ \cos(s)}{2}\rt)}}{\alpha -\beta\left(\frac{1+ \cos(s)}{2}\right)  -\sqrt{\left(\frac{1 - \cos(s)}{2}\right)\lt(1 - \beta^2\frac{1+ \cos(s)}{2}\rt)}} \times\frac{\alpha}{\alpha}\\
\\
& =  \frac{\alpha^2 - \left(\frac{1+ \cos(s)}{2}\right) + \sqrt{\left(\frac{1 - \cos(s)}{2}\right)\lt(\alpha^ 2 - \frac{1+ \cos(s)}{2}\rt)}}{\alpha^2 - \left(\frac{1+ \cos(s)}{2}\right) - \sqrt{\left(\frac{1 - \cos(s)}{2}\right)\lt(\alpha^ 2 - \frac{1+ \cos(s)}{2}\rt)}} \times \frac{\sqrt{\alpha^ 2 - \frac{1+ \cos(s)}{2}}}{\sqrt{\alpha^ 2 - \frac{1+ \cos(s)}{2}}}\\
% \displaybreak\\
& =  \frac{\sqrt{\alpha^2 - \frac{1 + \cos(s)}{2}} + \sqrt{\frac{1 - \cos(s)}{2}}}{\sqrt{\alpha^2 - \frac{1 + \cos(s)}{2}} - \sqrt{\frac{1 - \cos(s)}{2}}} \times \frac{\sqrt{\frac{1 - \cos(s)}{2}}}{\sqrt{\frac{1 - \cos(s)}{2}}}\\
\\
& =  \frac{\sqrt{\frac{2(\alpha^2 - 1)}{1 - \cos(s)}+1} + 1}{\sqrt{\frac{2(\alpha^2 - 1)}{1 - \cos(s)}+1} - 1}\\
\\
& =  1 + \frac{2}{\sqrt{\frac{2(\alpha^2 - 1)}{1 - \cos(s)}+1} - 1}
\end{align*}
\end{proof}

\begin{lem}\label{jerem}
 Considérons $M\in\NN^*$, $T_i >0$, pour $1\leq i \leq M$, et $\delta >0$ donnés. Il existe $t\geq 1$ et des entiers $k_1, \dots, k_{M}$ tels que pour tout $1\leq i \leq M$,
\[|k_i T_i - t|<\delta\]
\end{lem}
\begin{proof}
 Considérons le réseau de $\mathbb R^{M+1}$ engendré par les $(0,\dots,0,T_{n},$ \linebreak
 $0,\dots,0)$ (avec $T_n$ en $n^{\hbox{\scriptsize ième}}$ position) pour $n\leq M$ et par $(1,1,\dots,1)$, de volume fondamental $V$ le produit des $T_n$. Ainsi, en considérant le pavé $[-\delta,\delta]\times\dots\times[-\delta,\delta]\times[-V\delta^{-M},V\delta^{-M}]$, de volume $V 2^d$, on sait par le théorème de Minkowski qu'il contient au moins un point du réseau autre que l'origine. Les $M$ premières coordonnées de ce point sont de la forme $h_nT_n + h_{M+1}$ avec $h_j \in \mathbb Z$ pour $j\in \lin 1,M+1\rin$. Si $\delta < \min(T_1,\dots,T_{M},1)$ (et quitte à réduire $\delta$, nous supposons ceci satisfait), aucun de ces coefficients $h_j$ ne peut être nul, et nécessairement $h_{M+1}$ est de signe opposé aux autres $h_j$. Il suffit donc de prendre $k_n = |h_n|$ et $t=|h_{M+1}|$.
\end{proof}

\begin{lem}\label{decroi}
Notons, pour $s>0$ et $p\in]0,1[$,
\begin{eqnarray*}
 h(p) & = & \frac{p}{1-p^2}\frac{1+e^{-ps}}{1-e^{-ps}}\\
\\
\phi(p) & = & e^{ps}\left(1+\frac{2}{ph(p)+\sqrt{h(p)^2+\frac{1}{1-p^2}}-1}\right).
\end{eqnarray*}
Alors pour tout $s$, $p\mapsto \phi(p)$ est croissante. 

En prenant $p = \sqrt{1-\left(\frac{n}{a}\right)^2}$ et $s = 2at$ on obtient en particulier que pour $t>0$ et $a>0$, $n\in]0,a[ \mapsto R(t,a,n)$ est décroissante.
\end{lem}

\begin{proof}
 Le calcul de la dérivée est effectué \emph{via} Maple :
\begin{verbatim}
 h:=p->p/(1-p^2)*(1+exp(-s*p))/(1-exp(-s*p)):
phi:=p->exp(p*s)*(1+2/(p*h(p)+sqrt(h(p)^2+1/(1-p^2))-1)):
resultat := simplify(exp(-p*s)*diff(phi(p),p)):
\end{verbatim}
Le résultat est de la forme $\frac{\hbox{\scriptsize num\'erateur}(p)}{(\mathrm{un\  terme})^2(p^2-1)(e^{-ps}-1)}$ ; il s'agit donc de vérifier que le numérateur est positif. À l'instruction
\begin{verbatim}
 solve(numerateur(p)=0,p);
\end{verbatim}
la réponse est 
\begin{verbatim}
 -RootOf(_Z exp(_Z) + _Z + 2 - 2 exp(_Z))/s
\end{verbatim}
Autrement dit le numérateur s'annule en $p$ si $e^p = \frac{2+p}{2-p}$, équation dont la seule solution est $p=0$ : en effet, s'il y avait une autre solution $p^*$, la dérivée de $e^p \frac{2-p}{2+p}$ s'annulerait entre 0 et $p^*$, or celle-ci est $\frac{-z^2e^{z}}{(2+z)^2}$. Ainsi le numérateur est de signe constant pour $p\in]0,1[$ et $\phi$ est monotone. Les limites de $\phi$ en 0 et 1 sont respectivement $1 + \frac{2}{\sqrt{4s^{-2}+1}-1}$ et $e^{s}$, dont l'égalité est équivalente à $2e^s - 2 - s^2 = 0$, d'unique solution $s=0$ ; vu leurs équivalents pour $s\rightarrow+\infty$ on a donc $\phi(1)>\phi(0)$ pour $s>0$, donc $\phi$ est croissante.

\end{proof}

 \bigskip
 \par\hskip5mm\textbf{\large Remerciements :}\par\sm
 Nous sommes reconnaissant à Jérémy Leborgne
 pour l'élégant argument du lemme \ref{jerem}.

\end{document}